\documentclass{amsart}
\usepackage{amsrefs}
\usepackage{url}
\usepackage[colorlinks=true, linkcolor=black, citecolor=black]{hyperref}
\usepackage{amsmath, amssymb, amsthm}
\usepackage[english]{babel}
\usepackage{tikz-cd}

\newtheorem{thm}{Theorem}[section]
\newtheorem{cor}[thm]{Corollary}
\newtheorem{prop}[thm]{Proposition}
\newtheorem{lem}[thm]{Lemma}

\theoremstyle{definition}
\newtheorem*{conv}{Convention}
\newtheorem{question}[thm]{Question}

\title{Random dynamics of plane polynomial automorphisms}
\author{Arnaud Nerrière}
\date{}
\address{Université Bourgogne Europe, CNRS, IMB UMR 5584, 21000 Dijon, France}
\email{arnaud.nerriere@u-bourgogne.fr}

\begin{document}

\begin{abstract}
    Let $\mu$ be a finitely supported probability measure on the group of automorphisms of $\mathbb{A}^2_\mathbb{C}$. If the group generated by the support of $\mu$ is non-elementary and contains only loxodromic elements, we show the existence of dynamical Green functions associated to random products. We derive consequences for $\mu$-stationary measures: they are compactly supported, and we can apply Roda's theorem to show stiffness when the action is non-dissipative.
\end{abstract}

\maketitle
\tableofcontents

\section{Introduction}

Let $\mu$ be a finitely supported probability measure on the group $G$ of automorphisms of $\mathbb{A}^2_\mathbb{C}$. The data of $\mu$ defines a random dynamical system on $\mathbb{C}^2$. Denote by $\Sigma_\mu$ the support of $\mu$. We set $\Omega:=\Sigma_\mu^\mathbb{N}$, and if $\omega=(f_n)_{n \geq 0} \in \Omega$, we denote by $f^n_\omega$ the left product of the $n$ first terms of $\omega$, that is, $$f^n_\omega:=f_{n-1} \cdots f_0.$$

\subsection{Stationary measures}
Consider the empirical measures, defined by  $$\nu_N(\omega,p):=\frac{1}{N}\sum_{n=0}^{N-1} \delta_{f^n_\omega(p)}$$
for any $(\omega, p) \in \Omega \times \mathbb{C}^2 $.
We may have escape of mass: the sequence of empirical measures converges to zero for the weak topology when $ f^n_\omega(p)  $ goes to infinity. On the other hand, if $p$ is contained in a compact subset of the plane which is invariant by every automorphism of $\Sigma_\mu$, then by Breiman's law of large numbers, weak limits of $\nu_N(\omega,p)$ are almost surely $\mu$-stationary. Recall that a Borel probability measure $\nu$ on $\mathbb{C}^2$ is \textbf{$\mu$-stationary} if 
$$\sum_{f \in \Sigma_\mu} \mu(f)f_*\nu=\nu.$$

\begin{conv}
Throughout this paper, a measure on $\mathbb{C}^2$ will always mean a Borel probability measure.
\end{conv}

Our main goal is to classify $\mu$-stationary measures when the group $\Gamma_\mu:=\langle \Sigma_\mu \rangle$ generated by the support of $\mu$ is non-elementary. In this framework, we expect that the existence of a $\mu$-stationary measure is a rare phenomenon.

Jung's theorem \cite{Jung} asserts that the group $G$ of polynomial automorphisms of the plane is the amalgamated product of its affine and elementary subgroups. Any automorphism which is not conjugate to any elementary map is called \textbf{loxodromic}. A Hénon map $h(x,y)=(y,p(y)-ax)$ with $a \in \mathbb{C}\setminus \{0\}$ and $\deg(p) \geq 2$ is loxodromic. By a theorem of Friedland and Milnor \cite{FM}, every loxodromic automorphism is conjugate to a product of Hénon maps. Hénon maps have a rich dynamical behaviour and loxodromic automorphisms are precisely the elements of $G$ that have positive topological entropy. We say that $\mu$ is \textbf{non-elementary} if $\Gamma_\mu$ contains two loxodromic elements generating a free group of rank two, and \textbf{purely loxodromic} if $\Gamma_\mu \setminus \{1\}$ contains only loxodromic elements. 

Our first result shows that non-trivial recurrence can only appear in compact subsets of the plane.

\begin{thm}\label{thm2}
    Let $\mu$ be a finitely supported, non-elementary and purely loxodromic measure on $G$. If $\nu$ is a $\mu$-stationary measure on $\mathbb{C}^2$, then $\nu$ is compactly supported. 
\end{thm}

The assumptions in the above theorem are not very restrictive. Indeed, by a theorem of Lamy \cite{La3}, if $f$ and $g$ are two loxodromic elements that do not satisfy any relation of the form $f^n=g^m $ for some non-zero $n,m \in \mathbb{Z}$, then, there exists $ k \geq 1$ such that $f^k$ and $g^k$ generate a non-elementary, purely loxodromic subgroup of $G$. For any $d>1$, we denote by $G_d$ the variety of automorphisms of degree at most $d$. We have that, outside a countable union of hypersurfaces in $G_d\times G_d$, two automorphisms generate a free, purely loxodromic subgroup. See Section \ref{section2.2}.

We say that the random dynamical system $(\mathbb{C}^2,\mu)$ is \textbf{stiff} if every $\mu$-stationary measure $\nu$ on $\mathbb{C}^2$ is \textbf{$\Gamma_\mu$-invariant}, that is, $\gamma_*\nu=\nu$ for every $\gamma \in \Gamma_\mu$. We denote by $\mathrm{Jac}(f)\in \mathbb{C}$ the complex jacobian of an element $f \in G$. Notice that it is constant. We say that $\mu$ is \textbf {non-dissipative} if the random dynamical system is not volume decreasing, that is, $$\sum_{f \in \Sigma_\mu} \mu(f)\log |\mathrm{Jac}(f)|\geq  0.$$
Remark that if $\mu$ is symmetric, that is, $\mu(f)=\mu(f^{-1})$ for every $f \in G$, then $\mu$ is non-dissipative. Notice that we cannot expect stiffness to hold in the dissipative case. Indeed, one can consider a Hénon map $f$ having an attracting fixed point at the origin, and a small perturbation $g$ of $f$. Generically, $f$ and $g$ generate a non-elementary subgroup of $G$, and a small neighborhood of the origin will support a non-invariant $\mu$-stationary measure for $\mu=\frac{1}{2}(\delta_f + \delta_g)$. 
In the non-dissipative case, we have the following instance of stiffness.

\begin{thm}\label{thm3}
     Let $\mu$ be a finitely supported, non-elementary and purely loxodromic measure on $G$.
     If $\mu$ is non-dissipative, then every $\mu$-stationary measure is $\Gamma_\mu$-invariant.
\end{thm}

Thus, the classification of $\mu$-stationary measures boils down to the classification of $\Gamma_\mu$-invariant measures. Roda's theorem \cite{Roda} gives a further description of hyperbolic stationary measures.

\begin{thm}\label{thm4}
    Suppose $\mu$ is a finitely supported, non-elementary, purely loxodromic and non-dissipative probability measure on $G$. Then, every ergodic hyperbolic $\mu$-stationary measure on $\mathbb{C}^2$ is compactly supported, $\Gamma_\mu$-invariant and in one the following cases:
    \begin{enumerate}
        \item The support of $\nu$ is finite.
        \item The stable and unstable conditional measures are supported on real curves, they are absolutely continuous with respect to the Lebesgue measure on these curves, and jointly integrable.
    \item $\nu$ is absolutely continuous with respect to the Euclidean volume on $\mathbb{C}^2$.
    \end{enumerate}
\end{thm}

The case $(2)$ may happen for example if there are two Hénon maps defined over $\mathbb{R}$ preserving a Borel subset $B\subset \mathbb{R}^2$ of positive area. Note that if the stronger inequality $$\sum_{f \in \Sigma_\mu} \mu(f) \log | \mathrm{Jac}(f) | >0$$
is satisified, then every $\mu$-stationary hyperbolic ergodic measure is finitely supported. See \cite[Section 2.6.2]{Roda} and \cite[Section 2.6.3]{Roda}. To our knowledge, there is no known example of a $\Gamma_\mu$-invariant measure that is not finitely supported in this setting. 

\begin{question}
 Does there exists a non-elementary and non-dissipative finitely supported probability measure on $G$ such that $\mu$ admits a non-atomic $\mu$-stationary measure?    
\end{question}

Let us mention the related question, raised by Cantat and Dujardin \cite[Question 1.14]{CD6}: does there exists a non-elementary group $\langle f,g \rangle$ such that $f$ and $g$ fix the origin and are simultaneously conjugate to elements of $\mathrm{SU}_2(\mathbb{C})$ by a local biholomorphism? 

Under the additional hypothesis that the $\Gamma_\mu$-invariant measure has positive entropy for some $\gamma \in \Gamma_\mu$, we have the following classification result.

\begin{thm}\label{thm5}
    Let $f$ be a loxodromic element of $G$, and suppose that $\nu$ is an ergodic $f$-invariant measure such that $h_\nu(f) >0$. Suppose that the group $$G_\nu:=\{g \in G, \, g_* \nu =\nu \}$$
    is non-elementary. Then one of the following holds.
    \begin{enumerate}
        \item There exists a finitely supported, non-elementary and purely loxodromic measure $\mu$ on $G_\nu$, such that $\nu$ is hyperbolic as a $\mu$-stationary measure and the stable and unstable conditional measures are supported on real curves, and are absolutely continuous with respect to the Lebesgue measures on these curves, and jointly integrable.
         \item The group $G_\nu$ preserves the Euclidean volume on $\mathbb{C}^2$ and $\nu$ is the restriction of the volume to a subset of positive area which is $\Gamma_\mu$-invariant.
    \end{enumerate}
\end{thm}

\subsection{Green functions}
The above results rely on the study of the dynamics at infinity of the group $\Gamma_\mu$. We extend our random dynamical system to the projective plane $\mathbb{P}^2=\mathbb{C}^2 \cup L_\infty$ by birational transformations. The theorem of Friedland and Milnor \cite{FM} reduces the study of the dynamics of loxodromic elements of $G$ to the study of the dynamics of Hénon maps. One crucial fact is that a Hénon map $f$ is algebraically stable on $\mathbb{P}^2$, that is, $\deg(f^n)=\deg(f)^n$ for every $n \geq 1$. Equivalently \cite{Si}, the line at infinity is never contracted by an iterate of $f $ to the indeterminacy point of $f $. Moreover, $f$ admits a super-attracting fixed point on the line at infinity, and its Green function $$G_f:=\lim \frac{1}{\deg(f^n)}\log^+\Vert f^n \Vert$$
is a well-defined plurisubharmonic, continuous function on $\mathbb{C}^2$ that plays an important role for the study of its dynamics \cites{BS, HO, FS, BLS, JXH}.

We would like to define Green functions associated to a random dynamical system, by 
    $$G_\omega:=\lim \frac{1}{\deg(f^n_\omega)} \log^+\Vert f^n_\omega \Vert.$$
One major issue is the lack of algebraic stability: the degree of $f^n_\omega$ is not equal in general to the product of the degrees of the $f_i$. Geometrically, this means that $f^n_\omega$ may contract the line at infinity to the point of indeterminacy of $f_{n+1}$. In general, we cannot expect to birationally conjugate the dynamics of $\Gamma_\mu$ so that it becomes algebraically stable.  
Still, we are able to show that Green functions are well-defined when $\mu$ is purely loxodromic.

\begin{thm}\label{thm1}
    Let $\mu$ be a finitely supported, non-elementary and purely loxodromic measure on $G$.
    \begin{enumerate}
    \item For $\mu^\mathbb{N}$-almost every $\omega \in \Omega$, $$G_\omega=\lim_{n \rightarrow +\infty} \frac{1}{\deg(f^n_\omega)} \log^+\Vert f^n_\omega \Vert  $$
    is a well-defined, plurisubharmonic and continuous function on $\mathbb{C}^2$.
    \item  There exists a finite subset $F$ of the line at infinity $L_\infty$ and for every $p\in F$, there exists an open subset $U(p)\subset \mathbb{C}^2$, such that the following hold.
    \begin{enumerate}
        \item The closure of $U(p)$ in $\mathbb{P}^2$ is a neighborhood of $L_\infty \setminus p$.
        \item For $\mu^\mathbb{N}$-almost every $\omega \in \Omega$, there exists $p(\omega) \in F$ such that  $$U(p(\omega)) \subset \{G_\omega >0\}.  $$
    \end{enumerate}
    \end{enumerate}    
\end{thm}

As a special case, we obtain that $f^n_\omega (q)$ goes to infinity when $q \in U(p(\omega))$. Moreover, the zero locus of $G_\omega$ is unbounded and thus the closure in $\mathbb{P}^2$ of the zero locus of $G_\omega$ intersects $L_\infty$ in one point $p(\omega),$ that is,  $$\overline{\{G_\omega=0\}} \cap L_\infty =p(\omega).$$

\begin{thm}\label{thm7}
    Under the assumptions of Theorem \ref{thm1}, we have  $G_\omega \neq G_{\omega'}$ for $\mu^\mathbb{N}\otimes \mu^\mathbb{N}$-almost every $(\omega, \omega')\in \Omega \times \Omega$.
\end{thm}

Theorem \ref{thm7} will be a key ingredient for the classification of stationary measures.

\subsection{Ingredients of the proofs}
We first explain the ideas of the proof of Theorem \ref
{thm1}. We use the description of $G$ as an amalgamated product $G=A\underset{\cap}{*}E$ to write $f^n_\omega$ as a product $$f^n_\omega=a_{l(n),n}e_{l(n),n} \cdots a_{1,n}e_{1,n}$$
of affine maps $a_{i,n} \in A\setminus E $ and elementary maps $e_{i,n} \in E\setminus A $. By a theorem of Maher and Tiozzo \cite{MT1}, the induced right random walk on the Bass-Serre tree of $G$ converges almost surely to the the boundary of the tree. This implies that the first $k$ elements in the above normal form are constant when $n$ is sufficiently large, that is, $a_{i,n}$ and $e_{i,n}$ do not depend on $n $ for $i \leq k$ and $n\geq t_k$. We recover a kind of algebraic stability: $a_ke_k \cdots a_1e_1$ contracts the line at infinity to a point which is different from the point of indeterminacy of $e_{k+1}$, but which may be very close to it. To avoid this second issue, we apply a theorem of Lamy \cite[Proposition 3.18]{La1} which says that the orbit $\Gamma_\mu(L_\infty)$ of the line at infinity is finite when $\Gamma_\mu$ is purely loxodromic. This gives a good generating set for $\Gamma_\mu$ and we can use classical filtration properties to show the existence of $G_\omega$.

For Theorem \ref{thm7}, we remark that if $p(\omega) \neq p(\omega')$, then $G_\omega \neq G_{\omega'}$ by Theorem \ref{thm1}.(2). Otherwise, we prove that after a finite number of blow-ups, we obtain a model $\pi\colon X \rightarrow \mathbb{P}^2$ on which the closure of $\{G_\omega=0\}$ intersects the divisor at infinity $X \setminus \pi^{-1}(\mathbb{C}^2)$ in one point $p_X(\omega)$ which is different from $p_X(\omega')$. This rely on a theorem of Maher and Tiozzo \cite[Theorem 1.8]{MT2} showing that the Poisson boundary of $(G, \mu)$ is equal to the boundary of the Bass-Serre tree of $G$ with the hitting measure.

Theorem \ref{thm2} is a direct corollary of Theorem \ref{thm7}: we use Birkhoff's ergodic theorem to show that a stationary measure puts no mass on any $U(p)$. Note that we only need the fact that $f^n_\omega(q)$ goes to infinity when $q \in U(p(\omega))$. Thus, a weaker control on the growth rate of orbits would be sufficient, but leaving the purely loxodromic case, we do not know how to prove such a property.

Theorem \ref{thm3} is a consequence of Theorem \ref{thm7} and Roda's theorem \cite{Roda}. Let $\nu$ be a $\mu$-stationary measure. Theorem \ref{thm2} allows us to apply tools from Pesin theory. Lyapunov exponents of $\nu$ are well-defined and satisfy $\lambda^+ + \lambda^-\geq 0$ by assumption. We only need to consider the hyperbolic case $\lambda^+>0>\lambda^-$ thanks to the invariance principle of Crauel \cite{Cr} and Avila--Viana \cite{Av}. In this case, the stable manifolds $W^s(\omega,q)$ are included in $\{G_\omega=0\}$ and thus they are not non-random by Theorem \ref{thm7}. We apply Roda's theorem \cite{Roda} to derive stiffness. Note that Roda's theorem is written in the framework of compact complex surfaces, but it also applies to compactly supported stationary measures on any complex surface, as long as the stable manifolds depend non-trivially on the itinerary $\omega$.

Note that the proofs rely on the finiteness of the support of $\mu$, but we expect that Theorems \ref{thm2}, \ref{thm3} and \ref{thm4} hold under optimal moment conditions.

\subsection{Related results}
Stiffness property was introduced by Furstenberg \cite{Furstenberg} in the framework of homogeneous dynamics. In a seminal work, Benoist--Quint \cite{BQ} have given a classification of stationary measures in this context. Since then, their ideas have been extended to increasingly more general frameworks: Eskin--Mirzakhani \cite{EM} for $\mathrm{SL}_2(\mathbb{R})$-actions on moduli space of translation surfaces, Brown--Rodriguez Hertz \cite{BRH} for smooth actions on real surfaces, and Brown--Eskin--Filip--Rodriguez Hertz \cite{BEFRH} for smooth actions in any dimension. See the introduction of \cite{CD1} for a more detailed historical account.

Inspired by the work of Brown and Rodriguez Hertz, Cantat--Dujardin \cite{CD1} have studied random dynamics on compact complex surfaces. They establish stiffness on the real part $X(\mathbb{R})$ of a K3 surface $X$ defined over the reals when $\mu$ is a finitely supported measure on the group of automorphisms of $X$ such that $\Gamma_\mu$ does not preserve any curve. Building on the results of Cantat--Dujardin and  Brown--Eskin--Filip--Rodriguez Hertz, Roda \cite{Roda} has extended the stiffness property for groups of automorphisms of complex K3 surfaces that do not preserve any curve, that is, stiffness holds also on $X(\mathbb{C})$. If a $\mu$-stationary measure $\nu$ is hyperbolic, then the stable manifolds $W^s(\omega, q)$ are almost surely complex one-dimensional submanifolds biholomorphic to $\mathbb{C}$. When $\Gamma_\mu$ does not preserve any curve, these stable manifolds depend non-trivially on the itinerary $\omega$. This property is crucial in order to apply stiffness results from \cites{BRH, BEFRH}. Indeed, Cantat and Dujardin have shown that the cohomology classes of positive closed currents associated to the stable manifolds are almost surely different in $H^{1,1}(X)$ using a random walk on a finite dimensional hyperbolic space $\mathbb{H}\subset H^{1,1}(X,\mathbb{R}
)$. Our proof follows the same strategy: we show that stable manifolds are not non-random using the random walk on the Bass-Serre tree to distinguish the $W^s(\omega,q)$ at infinity. It allows us to apply Roda's deep theorem. Cantat--Dujardin \cite{CD2} have obtained further classification results when the group $\Gamma_\mu$ contains parabolic elements. See \cite{CD6} for a survey of their results, examples and open problems. 

In the case of affine surfaces, which are non-compact, Cantat--Dupont--Martin-Baillon \cite{CDMB} have classified stationary measures when $X$ is a Markov surface. This framework is similar to $\mathbb{C}^2$ because one has to deal with the dynamics of $\text{Aut}(X)$ at infinity, but the situation is more rigid for Markov surfaces. Indeed, their automorphism group is a free product of three copies of $\mathbb{Z}/2\mathbb{Z}$, and the indeterminacy points of elements of $\text{Aut}(X)$ are contained in a finite subset of $\partial X$. Thus, the issue that $f^n_\omega$ may contract an exceptional curve to a point which is very close to a point of indeterminacy of $f_{n+1}$ never happens in the context of Markov surfaces. Moreover, compact $\text{Aut}(X)$-invariant subsets were already classified in \cite{CL} whereas there is no classification of compact invariant subsets for a non-elementary subgroup of $G$ apart from some particular cases. See also \cite{Goldman,RR} for the study of the dynamics on Markov surfaces.

In his thesis, Lamy has studied the action of subgroups of $G$ on the Bass-Serre tree, and the dynamics at infinity in $\mathbb{C}^2$ of purely loxodromic subgroups of $G$, see \cite[Section 3.2]{La1}. Notably, \cite[Proposition 3.18]{La1} gives good generating sets for purely loxodromic subgroups and is a key ingredient for our results. 

The compactness of the support of stationary measures has been proved recently by Quintero Santander \cite{QS} when $\Gamma_\mu$ is a ping-pong group, that is, a subgroup of $G$ generated by a Hénon map $f$ and a conjugate of $f$ by a rotation which does not preserve the set $\{[1:0:0],[0:1:0]\}$. The analysis of the dynamics at infinity is simplifed by the fact that the indeterminacy points of the elements of $\Sigma_\mu$ are different. See Section \ref{section2.2}. Actually, we proved the same result independentely, and one goal of this paper is to extend it to the more general case of purely loxodromic groups.

Note also that Bera \cite{Bera1} has constructed Green functions for semigroups of Hénon maps. In this setting, the point $[0:1:0]$ is uniformly attracting for the semigroup and there is no issue of algebraic stability. These Green functions play an important role in her proof with Verma of Bedford's conjecture \cite{Bera2}. 

There is also a connection to the work of Diller and Roeder \cite{DR}, where they prove equidistribution results for iterated preimages of curves under toric rational maps that do not admit any algebraically stable model. Indeed, in our setting, we can show that iterated pullbacks of the Fubini-Study form converge almost surely to the current $T_\omega:=dd^cG_\omega$. We can not directly use the argument of Cantat--Dujardin to distinguish the stable currents in our setting  because $H^{1,1}(\mathbb{P}^2)$ is too small. Thus, we would have to consider the action on the Picard-Manin space, which contains an infinite dimensional hyperbolic space, to distinguish the cohomology classes of the currents $T_\omega$. This implies to work with currents on all models of the projective plane as in \cite{DR}. We do not develop this idea in this paper, as for automorphisms of the plane, the Bass-Serre tree works equally well. It will be the subject of a companion paper, in which we prove equidistribution results for certain random dynamical systems on $\text{Bir}(\mathbb{P}^2)$.

\subsection{Organization of the paper}
We use convergence properties of the random walk on the Bass-Serre tree of $G$ to obtain normal forms for $f^n_\omega$ in Section 2. Then, we define the Green functions $G_\omega$ in Section 3. We introduce dynamics on blow-ups of $\mathbb{P}^2$ and we prove Theorems \ref{thm7} and \ref{thm2} in Section 4. The last section contains the proofs of Theorems \ref{thm3}, \ref{thm4}, \ref{thm5} and further classification results.

\subsection{Acknowledgments}
This work was partially funded by the French National Research Agency under the project DynAtrois (ANR-24-CE40-1163), by the EIPHI Graduate School (contract ANR-17-EURE-0002), and by the Région Bourgogne-Franche-Comté. The author would like to thank Serge Cantat, Romain Dujardin, Christophe Dupont, Florestan Martin-Baillon, Andres Quintero Santander and Roland Roeder for interesting discussions related to this paper, and his Ph.D. advisors Johan Taflin and Michele Triestino for their support. 

\section{Random walk on the Bass-Serre tree}
In this section, we first recall the definition of the Bass-Serre tree of $G$ and we describe Lamy's results \cite{La3} on the classification of subgroups of $G$. Then, we state a theorem of Maher and Tiozzo \cites{MT1, MT2} concerning the asymptotic behaviour of random walks on Gromov hyperbolic spaces. Applying this to the random walk on the Bass-Serre tree, we derive normal forms for random products $f^n_\omega$. We conclude by a finiteness result of Lamy \cite{La1} which gives a good generating set for purely loxodromic subgroups of $G$.

\subsection{The Bass-Serre tree}\label{section2.1}
We denote by $G$ the group of automorphisms of the complex affine plane $\mathbb{A}^2_\mathbb{C}$. We consider the affine group
$$A=\{a(x,y)=(a_1x+b_1y+c_1,a_2 x+b_2 y+ c_2)|\; a_1b_2-a_2b_1 \neq 0\},$$
and the \textbf{elementary} group
$$E=\{e(x,y)=(\alpha x+p(y), \beta y+\delta)|\; \alpha \beta \neq 0, \, p\in \mathbb{C}[y]\}.$$
Denote by $S:=A\cap E$  their intersection.
By Jung's theorem \cite{Jung}, $G$ is the amalgamated product of its two subgroups $A$ and $E$ over $S$. See \cite{La4} for a modern proof. Every element $g \in G$ which is not in $S$ can be written as a product $$g=a_1e_1\dots a_n e_n$$
with $a_i \in A\setminus E$, $e_i\in E \setminus A$, with possibly $a_1=1$ and $e_n=1$. This writing is unique up to the modifications $a_i'=a_is$, $e_i'=s^{-1}e_i$ or $e_i'=e_is$, $a_{i+1}'=s^{-1}a_{i+1}$ for some $s \in S$.

Recall that the birational extension to $\mathbb{P}^2=\mathbb{C}^2\cup L_\infty$ of an elementary map $e \in E\setminus A$ has an indeterminacy point at $I:=[1:0:0]$, and $e$ contracts $L_\infty$ to $I$. An affine map $a \in A \setminus E$ extends to an automorphism of $\mathbb{P}^2$ sending $I$ to a point $a(I) \in L_\infty$ which is different from $I$.

We define a \textbf{Hénon map} by $$h(x,y)=(y,p(y)-ax)$$
where $a \in \mathbb{C} \setminus \{0\}$ and $\deg(p) \geq 2$. Note that we can decompose $h=a e$ as a product of the affine map $a(x,y)=(y,x)$ and the elementary map $e(x,y)=(p(y)-ax,y)$. By a theorem of Friedland and Milnor \cite{FM}, every $g \in G$ is conjugate to an elementary map or to a product of Hénon maps. A Hénon map has an indeterminacy point at $I$ and contracts $L_\infty$ to the super-attracting fixed point $[0:1:0]$.

We now define the Bass-Serre tree of $G$ which is a simplicial tree $\mathcal{T}$ on which $G$ acts faithfully by isometries \cite{Serre, La3}. The set of vertices of $\mathcal{T}$ is $G/A \cup G/E$, and there is an edge between the left classes $gA$ and $g'E$ if and only if there exists $h\in G$ such that $hA=gA$ and $hE=g'E$. Notice that if $h'$ satisfies the same property, then $h'=hs$ for some $s\in S$. Thus, the set of edges of $\mathcal{T}$ is represented by $G/S$. There are no cycles thanks to the amalgamated product property. We first endow $\mathcal{T}$ with the natural distance $d_\mathcal{T}$ of a graph and we then consider the geometric realization of $\mathcal{T}$ where an edge is isometric to $[0,1]$ with the euclidean distance. Thus, $\mathcal{T}$ is path-connected, and there is a unique geodesic between two points. Notice that $\mathcal{T} $ is not locally compact: there is an uncountable set of edges starting from one vertex, see \cite[Section 2]{La3} for a description of representatives of $G/A$ and $G/E$. 

The action by left translations $f(gS):=fgS$ defines a representation of $G$ in the isometry group of $\mathcal{T}$. The action is faithful by a theorem of Lamy \cite[Remarque 3.5]{La3}. An element $g \in G$ is \textbf{elliptic} if $g$ fixes a point in $\mathcal{T}$. Otherwise, $g$ is called \textbf{loxodromic}. In this case, $g$ acts on $\mathcal{T}$ by translation along an infinite geodesic, called the \textbf{axis} of $g$. An automorphism is elliptic precisely when it is conjugate to an elementary map. Thus, for an automorphism $g$, the following are equivalent \cite{FM}, \cite{Smillie}:
\begin{enumerate}
    \item $g$ is loxodromic,
    \item $g$ is conjugate to a product of Hénon maps,
    \item the dynamical degree $ \lambda(g):=\lim (\deg(g^n))^\frac{1}{n}$ of $g$ satisfies $\lambda(g) >1$,
    \item the dynamical degree $\lambda(g)$ of $g$ is in $\mathbb{N}_{\geq 2}$,
    \item $g$ has positive topological entropy.
\end{enumerate}

We say that a subgroup $\Gamma$ of $G$ is \textbf{non-elementary} if it contains two loxodromic elements that generate a free group of rank two and is \textbf{purely loxodromic} if every element of $\Gamma \setminus  \{1\}$ is loxodromic. We have a classification of subgroups of $G$, which is due to Lamy \cite{La3}. See also \cite[Chapter 7]{La2}.

\begin{prop}[Lamy]\label{prop1}
For a finitely generated subgroup $\Gamma \le G$, exactly one of the following holds.
    \begin{enumerate}
    \item $\Gamma$ does not contain any loxodromic element and is conjugate to a subgroup of $A$ or $E$.
    \item $\Gamma$ contains loxodromic elements and they have the same axis. In this case, $\Gamma$ is virtually $\mathbb{Z}$.
    \item $\Gamma$ contains two loxodromic elements with different axis. In this case, $\Gamma$ contains a free group of rank two which is purely loxodromic. 
\end{enumerate}
Moreover, if $f$ and $g$ are loxodromic elements that do not satisfy any relation of the form $f^n=g^m$ with non-zero $n,m \in \mathbb{Z}$, then there exists $k \geq 1$, such that $f^k$ and $g^k$ generate a free group of rank two which is purely loxodromic. 
\end{prop}

We say that a measure $\mu$ is non-elementary (\textit{resp.\,}purely loxodromic) if the subgroup $\Gamma_\mu$ generated by the support of $\mu$ is non-elementary (\textit{resp.\,}purely loxodromic). We see that the assumptions in Theorem \ref{thm1} are not too restrictive: if $f$ and $g$ are loxodromic, then, unless they share common powers, there exists $k\geq 1$ such that a measure supported on $f^k, g^k $ and their inverses is non-elementary and purely loxodromic.

\subsection{Genericity and typical examples}\label{section2.2}
By \cite[Corollary 1.6]{BCW}, the subset $G_d \subset G$ of elements of $G$ that have degree at most $d$ is an open subset of an affine algebraic variety. If $w(f,g)$ denotes a reduced word in $(f,g)\in G_d \times G_d$, then the conditions $w(f,g)=\mathrm{id}$ define a countable union of proper subvarieties. Moreover, the conditions describing when $w(f,g) $ is elliptic are also given by polynomial equations in the coefficients of $f$ and $g$ by Furter's criterion \cite{Furter}. We obtain the following: if $f$ and $g$ are outside a countable union of hypersurfaces of $G_d \times G_d$, then they generate a free, purely loxodromic subgroup of $G$.    

A typical family of purely loxodromic and non-elementary measures can be constructed as follows: consider two loxodromic automorphisms $f$ and $g$ such that $$\# \{I_f, I_{f^{-1}}, I_g, I_{g^{-1}} \}=4,$$
where $I_\gamma$ denotes the indeterminacy point of $\gamma$ on $\mathbb{P}^2$. For instance, one can take $f$ to be a Hénon map and $g$ to be the conjugate of $f$ by an affine map $a$ satisfying $$\{a(I_{f}), a(I_{f^{-1}})\} \cap \{I_f, I_{f^{-1}} \} = \emptyset.$$  
Then, the measure $$\mu:=\frac{1}{4}(\delta_f+\delta_{f^{-1}}+\delta_g+\delta_{g^{-1}})
$$ satisfies the assumptions of Theorem \ref{thm1}. Indeed, one can use a ping-pong argument at infinity in $\mathbb{P}^2$ to show that $\Gamma_\mu$ is a free group which is purely loxodromic. We refer to the work of Quintero Santander \cite{QS} for a detailed description of the random dynamics in this setting.

\subsection{Random walks on hyperbolic spaces} Random walks on locally compact trees are well understood, and under reasonable assumptions on the law, we obtain classical properties such as transience, convergence to the boundary at linear speed and identification of the Poisson boundary. Here, we cannot use these results as we are working with the Bass-Serre tree which is not locally compact. We state a theorem of Maher and Tiozzo which combines results from \cites{MT1, MT2} that concerns random walks on non-proper Gromov hyperbolic spaces. 

Let $(X,d)$ be a Gromov hyperbolic space, and consider a finitely supported probability measure $\mu$ on a subgroup $\Gamma $ of the isometry group of $X$. We denote by $\Sigma_\mu$ the support of $\mu$ and we set $\Omega:=\Sigma_\mu^\mathbb{N}$. Let $x_0 \in X$ be any base point. For $\omega=(f_n)_{n\geq 0},$ we define the right random walk on $X$ based on $x_0$ by $r^0_\omega x_0:=x_0$ and $$ r^n_\omega x_0:=f_0 \cdots f_{n-1}(x_0)$$
when $n\geq 1$. 
Isometries of $X$ are classified as elliptic, parabolic or loxodromic \cite{DSU}. Loxodromic elements act with two fixed points on $\partial X$, one attracting, one repelling. We say that $\mu$ is non-elementary if the semigroup generated by the support of $\mu$ contains two loxodromic elements with disjoint fixed point sets in $\partial X$. We say that a loxodromic element $f \in \Gamma$ is weakly properly discontinuous (WPD) if for any $x \in X$ and $K\geq 0$, there exists $N\geq 1$, such that the set
$$\{ \gamma \in \Gamma \mid \, d(x, \gamma (x)) \leq K \, , \, d(f^N(x), \gamma f^N(x)) \leq K \}  $$
is finite. The idea is that $f$ is WPD if the action of $\Gamma$ along the axis of $f$ is properly discontinuous. This weak form of properness is needed in order to identify the Poisson boundary of the walk for groups acting on non-proper Gromov-hyperbolic spaces. We say that $\mu$ is WPD if the semigroup generated by $\Sigma_\mu$ contains a WPD element.

\begin{thm}[Maher--Tiozzo]\label{prop3}
    Let $\mu$ be a finitely supported, non-elementary probability measure on a subgroup $\Gamma$ of isometries of a Gromov hyperbolic space $(X,d)$, and consider a base point $x_0 \in X$.
    \begin{enumerate}
        \item For $\mu^\mathbb{N}$-almost every $\omega$, there exists a point $\xi (\omega) \in \partial X$ such that $$r^n_\omega x_0 \longrightarrow \xi(\omega) $$ as $n \to +\infty.$
        \item There exists $l>0$, such that for $\mu^\mathbb{N}$-almost every $\omega$, $$\lim_{n \rightarrow +\infty} \frac{1}{n}d(r^n_\omega x_0 ,x_0) =l.$$
        \item We endow the Gromov boundary $\partial X$ with the hitting measure $\eta $ defined by 
        $$\eta(B):=\mu^\mathbb{N}(\omega, \, \xi(\omega) \in B)$$
        for any Borel subset $B \subset \partial X$. The hitting measure is non-atomic, and if we assume moreover that $\mu$ is WPD, then $(\partial X, \eta)$ is a model for the Poisson boundary of the right random walk $(\Gamma,\mu)$, that is, the map $$\phi \mapsto \left( \gamma  \mapsto\int_{\partial X } \phi(\gamma (b)) \, d\eta(b)   \right) $$ 
        defines an isomorphism between $L^\infty(\partial X, \eta)$ and the space $H^\infty(\Gamma,\mu)$ of bounded $\mu$-harmonic functions.
    \end{enumerate}
\end{thm}

\begin{proof}
    The first part comes from \cite[Theorem 1.1]{MT1}, and the second from \cite[Theorem 1.2]{MT1}. 
    See \cite[Theorem 1.8]{MT2} for the third part.  
\end{proof}

As the Bass-Serre tree of $G$ is a $0$-hyperbolic space, with an action of $G$ by isometries, we can apply Maher-Tiozzo's theorem to our setting. Remark that in this case the definition of a non-elementary measure $\mu$ coincides with the definition given in Section \ref{section2.1}, that is, the subgroup generated by the support of $\mu$ contains two loxodromic elements with disjoint fixed point sets on $\partial \mathcal{T}$. Indeed, if the semigroup generated by the support of $\mu$ fixes a point on the Bass-Serre tree, or preserves a set of two points in $\partial \mathcal{T}$, then the subgroup generated by the support of $\mu$ also fixes a point of $\mathcal{T}$ or preserves a set of two points in $\partial \mathcal{T}$. 

\begin{lem}\label{lemme2}
    Any loxodromic element in $G$ is WPD.
\end{lem}

\begin{proof}
See \cite[Proposition 5.4]{LL}.    
\end{proof}

\subsection{Normal forms} Consider a finitely supported, non-elementary probability measure $\mu$ on $G$. The following proposition asserts that the first $k$ elements that appear in the decomposition of $f^n_\omega$ as a product of affine and elementary maps may be chosen to be constant when $n $ is sufficiently large.

\begin{prop}\label{prop16}
    Consider a finitely supported, non-elementary, probability measure $\mu$  on $G$. For $\mu^\mathbb{N}$-almost every $\omega \in \Omega$, there exist sequences
    \begin{itemize}
        \item $t_\omega=(t_k)_{k \geq 1} \in \mathbb{N}^\mathbb{N}$, with $t_{k+1} >t_k$ for any $k$,
        \item $a_\omega=(a_k)_{k \geq 1} \in (A\setminus E)^\mathbb{N}$,
        \item  $ e_\omega=(e_k)_{k\geq 1}\in (E\setminus A)^\mathbb{N}$,
    \end{itemize}
    such that: for every $k\geq 1$ and $n\geq t_k$, we can write $$f^n_\omega=a_{l(n),n}e_{l(n),n}\cdots a_{1,n}e_{1,n}a_ke_k \cdots a_1e_1 $$
    with $l(n)\geq 0$ and $a_{i,n} \in A\setminus E, e_{i,n}\in E\setminus A$. This writing may start with an affine and end with an elementary.
\end{prop}

\begin{proof}
     We define the reflected measure $\check \mu$ by $\check\mu(g):=\mu(g^{-1})$. Note that $\check \mu$ is also non-elementary. Applying Maher--Tiozzo's theorem to the right random walk on $\mathcal{T}$ of law $\check \mu$ based on (any point of) $S$ yields the following: for $\check\mu ^\mathbb{N}$-almost every $\omega=(f_n)_{n \geq 0}$, there exist sequences
\begin{itemize}
    \item $(t_k)_{k \geq 1} \in \mathbb{N}^\mathbb{N}$ with $ t_{k+1}>t_{k}$ for any $k$,
    \item $(a_k)_{k \geq 1} \in (A \setminus E)^\mathbb{N}$,
    \item $(e_k)_{k \geq 1} \in (E \setminus A)^\mathbb{N}$,
\end{itemize}
such that, for every $k \geq 1$, and $n \geq t_k$, we can write $$r^n_\omega =a_1 e_1 \cdots a_k e_k a_{1,n} e_{1,n} \cdots a_{l(n),n} e_{l(n),n}$$
with $l(n) \geq 0$ and $a_{i,n} \in A \setminus E, e_{i,n} \in E \setminus A$. This writing may start with an elementary map and end with an affine map. It suffices to take inverses to obtain the proposition.
\end{proof}

Remark that the time $t_k$ is the time when the walk definitely leaves the ball of radius $2k$ centered on the edge $S$. Moreover, $2l(n)$ represents the distance of the walk to the geodesic ray $\xi(\omega)=(S,a_1S, a_1e_1S, \dots)$. 
We have the following instance of algebraic stability.

\begin{lem}\label{lemme1}
    Consider the normal form given by Proposition \ref{prop16}. Then, we have $$\deg(f^n_\omega)=\prod_{j=1}^{l(n)} \deg(e_{j,n}) \prod_{i=1}^k \deg(e_i).$$
\end{lem}

\begin{proof}
    We can assume that the normal form of $f^n_\omega$ starts with an elementary map. Recall that an elementary map $e \in E \setminus A$ has indeterminacy point $I:=[1:0:0]$ and contracts $L_\infty$ to $I$. Recall also that an affine map $a \in A \setminus E$ does not fix $I$. Thus, $a_1e_1$ contracts $L_\infty $ to a point which is not a point of indeterminacy for $a_2e_2$. We see inductively that the degree of the product of the maps $a_i e_i$ is equal to the product of their degrees.
\end{proof}

We would like to use these normal forms to define Green functions, but we do not know how to get a sufficient control on the dynamics at infinity in this general setting.

\subsection{A finiteness result} We will improve Proposition \ref{prop16} in the purely loxodromic case thanks to the following proposition, due to Lamy \cite{La1}.

\begin{prop}\label{prop2}
    Let $\Gamma$ be a finitely generated, purely loxodromic subgroup of $G$. There exist finite families $\mathcal{A} \subset A \setminus E$ and $\mathcal{E}\subset E \setminus A$, such that every $g \in \Gamma \setminus \{1\}$ can be written as a product
    $$g=a_n e_n \cdots a_1 e_1$$
    with $a_i \in\mathcal A$, $e_i \in \mathcal{E} $ and this writing may start with an affine map and end with an elementary map. Moreover, we can choose $\mathcal{A}$ and $\mathcal{E}$ symmetric.
\end{prop}

\begin{proof}
The idea of the proof is to show that $\Gamma$ admits a generating set $\Sigma$ satisfying the following property: if one writes $f \in \Sigma$ as a product of elements of $A\setminus E $ and $E\setminus A$, then there are only finitely many simplifications that appear among every product $f_1\cdots f_n$ of elements in $\Sigma$.
See \cite[Proposition 3.18]{La1} for the details of the proof.
\end{proof}

By Proposition \ref{prop2}, we can upgrade Proposition \ref{prop16} to the following result. 

\begin{prop}\label{prop17}
    Consider a finitely supported, non-elementary, purely loxodromic, probability measure $\mu$ on $G$. Consider the finite families  $\mathcal{A} \subset A\setminus E$ and $\mathcal{E} \subset E\setminus A $  given by Propostion \ref{prop2}. For $\mu^\mathbb{N}$-almost every $\omega \in \Omega$, there exist sequences
    \begin{itemize}
        \item $t_\omega=(t_k)_{k \geq 1} \in \mathbb{N}^\mathbb{N}$, with $t_{k+1} >t_k$ for any $k$,
        \item $a_\omega=(a_k)_{k \geq 1} \in \mathcal{A}^\mathbb{N}$,
        \item  $ e_\omega=(e_k)_{k\geq 1}\in \mathcal{E}^\mathbb{N}$,
    \end{itemize}
    such that: for every $k\geq 1$ and $n\geq t_k$, we can write $$f^n_\omega=a_{l(n),n}e_{l(n),n}\cdots a_{1,n}e_{1,n}a_ke_k \cdots a_1e_1 $$
    with $l(n)\geq 0$ and $a_{i,n} \in \mathcal
    A, e_{i,n}\in \mathcal{E}$. This writing may start with an affine map and end with an elementary map.
\end{prop}

Notice the difference with Proposition \ref{prop16}: we only need finitely many affine and elementary automorphisms to write any $f^n_\omega$.

\subsection{Asymptotic entropy}
 We give a corollary of Maher--Tiozzo's theorem that will be a key ingredient in the proof of Theorem \ref{thm7}. 
 
\begin{cor}\label{cor3}
    Consider a finitely supported, non-elementary, probability measure $\mu$ on $G$. There exists $h>0$, such that if $A_n $ are subsets of $\Gamma_\mu$ satisfying $$| A_n| \leq Ce^{n(h-\epsilon)},$$ for some $\epsilon>0$ then 
        $$\mu^\mathbf{N}(\limsup r^n_\omega  \in A_n)=0,$$
    that is, for $\mu^\mathbb{N}$-almost every $\omega \in \Omega$, the right walk $r^n_\omega$ is located in the subset $A_n$ only for finitely many times.
\end{cor}

\begin{proof}
    We denote by $\mu^{*n}$ the $n$-th convolution of $\mu$ with itself, that is, the law of the product of $n$ independent random variables $g_1, \dots, g_n$ of law $\mu$. By Kingman's theorem, the asymptotic entropy of the random walk  $$h:=\lim \frac{-\log \mu^{*n}(r^n_\omega )}{n}$$
    is well-defined $\mu^\mathbb{N}$-almost surely. Moreover, by Kaimanovich--Vershik's criterion \cite{KV}, $h$ is positive due to the non-triviality of the Poisson boundary. Consider the event $$\Omega_n:=\{\omega \in \Omega \mid \mu^{*n}(r^n_\omega) \leq e^{-n(h-\frac{\epsilon}{2})}\}.$$ Then, $$\mu^\mathbb{N}(\omega \in \Omega_n, r^n_\omega \in A_n) \leq Ce^{-n\frac{\epsilon}{2}}.$$
    Thus, Borel-Cantelli's lemma implies that
    $$\mu^\mathbb{N}(\limsup (\omega \in \Omega_n, r^n_\omega \in A_n))=0. $$
    We have also $$\mu^\mathbb{N}(\liminf \Omega_n)=1$$ and this concludes the proof.
\end{proof}

\section{Filtration and Green functions}\label{section3}

Let $\mu$ be a finitely supported, non-elementary and purely loxodromic probability measure on $G$. We study the dynamics at infinity of the random dynamical system $(\mathbb{C}^2, \mu)$ using the normal form obtained in Proposition \ref{prop17}. We use classical filtration properties to show that Green functions are well-defined. We refer to \cite{GZ} for basics on pluripotential theory.  

\subsection{Filtration}
Recall that a Hénon map $f(x,y)=(y, p(y)-ax)$ has the following dynamics at infinity: its indeterminacy point is  $I:=[1:0:0]$ and $f$ contracts the line at infinity $L_\infty$ to the point $[0:1:0]$ which is a super-attracting fixed point for $f$. One can show that a forward orbit is either bounded or is in the super-attracting basin and thus the Green function of $f $, $$G_f:=\lim \frac{1}{\deg(f^n)} \log^+\Vert f^n \Vert$$
is well-defined on the plane \cites{BS, DS, Abboud2}.  
Consider now the finite families $\mathcal {A} \subset A\setminus E$ and $\mathcal{E} \subset E \setminus A$ given by Proposition \ref{prop2} and define  $$\mathcal{H}:=\mathcal{A}\mathcal{E}=\{ h=ae \mid \, a\in \mathcal{A}, e\in \mathcal{E} \}.$$
Any element $h \in \mathcal{H}$ has indeterminacy point $I$ and contracts the line at infinity to the point $a(I) \neq I$. By Proposition \ref{prop17}, we can write $f^n_\omega$ as a product of elements of $\mathcal{H}$ up to multiplication by an affine map on the left and on the right. We will find a neighborhood of $L_\infty \setminus I$ which is invariant by every element of $\mathcal{H}$, and on which we have a good control on the growth rate of orbits of the semigroup generated by $\mathcal{H}$.

Denote by $\Vert \cdot \Vert $ the sup norm in $\mathbb{C}^2$.
We consider, for $R \geq 1$ and $\epsilon \in (0,1)$ the subsets 
    \begin{align*}
        V_{R,\epsilon}^+ &:= \{(x,y) \in \mathbb{C}^2 \mid \, \Vert (x,y) \Vert \geq R \; \text{and} \; |y| \geq \epsilon |x| \, \},\\
        V_{R,\epsilon}^- &:= \{(x,y) \in \mathbb{C}^2 \mid \, \Vert (x,y) \Vert \geq R \; \text{and} \; |x| > \epsilon^{-1} |y| \, \},
    \end{align*}
so that we have a partition $\mathbb{C}^2 \setminus B(0,R)=V_{R,\epsilon}^+ \cup V_{R,\epsilon}^-.$ Notice that the family of closures of $V^-_{R,\epsilon}$ in $\mathbb{P}^2$ form a basis of neighborhoods of the point $I$.

\begin{prop}\label{prop6}
        There exists $\epsilon_0>0$ such that for every $\epsilon \in (0,\epsilon_0)$, there exist $R_\epsilon \geq 1$, $C_\epsilon >0$ such that for every $h \in \mathcal{H} $ and $R\geq R_\epsilon$, the following hold. 
        \begin{enumerate}
            \item $h(V_{R,\epsilon}^+) \subset V_{R,\epsilon}^+$.
            \item $\Vert h(q)\Vert \geq C_\epsilon \Vert q \Vert^{\deg(h)}$ for any $q \in V_{R,\epsilon}^+.$
            \item $h^{-1}(V_{R,\epsilon}^-) \subset V_{R,\epsilon}^-.$
        \end{enumerate}
\end{prop}

\begin{proof}
    Consider an elementary map in $\mathcal{E}$, say $e(x,y)=(\alpha x +p(y), \beta y + \delta)$ with $\deg(p) \geq 2$ and $\alpha \beta \neq 0$. If $\epsilon >0$ is fixed, then, for $R$ large enough and $(x,y) \in V_{R,\epsilon }^+$, we have 
    \begin{align*}
            |\alpha x + p(y)| &\geq |p(y)| - |\alpha||x|\\
            & \geq |p(y)| - \frac{|\alpha|}{\epsilon} |y| \\
            &\geq C |y|^{\deg(p)} \\
            &\geq C \epsilon^{\deg(p)} \Vert (x,y) \Vert^{\deg(p)}.
    \end{align*}
    for some $C>0$. 
    Up to increase $R$, we have $|\alpha x + p(y)| \geq \epsilon^{-1} |\beta y + \delta|,$ and we obtain  $$e(V_{R,\epsilon}^+) \subset V_{C(\epsilon R)^{\deg(p)}, \epsilon}^-.$$
    If $a \in \mathcal{A}$, then up to increase $R$ again, there exists $C_a>0$ such that $\Vert a(q) \Vert \geq C_a \Vert q \Vert  $
    when $\Vert q \Vert \geq R$. 
    We can also assume $C(\epsilon R)^{\deg(p)} \geq R$, and we obtain 
    \begin{align*}
        \Vert h(q) \Vert &\geq C_a \Vert e(q) \Vert \\
        &\geq C_a C \epsilon^{\deg(p)} \Vert q \Vert^{\deg(p)}.
    \end{align*}
    Moreover, as $I$ is not fixed by $a$, we can choose $\epsilon>0$ small enough, and $R\geq R_\epsilon$ large enough, such that
    $$a(V_{C(\epsilon R)^{\deg(p)},\epsilon}^-)\subset V_{R,\epsilon }^+.$$
    This concludes the proof for one element $h=ae $ in $\mathcal{H}$ and we can choose uniform constants as $\mathcal{H} $ is finite. We write $h^{-1}=e^{-1}a^{-1}$, and the third part follows from the same proof as above and the fact that $\mathcal{A}$ and $\mathcal{E}$ are symmetric.
\end{proof}

Considering also upper bounds, we obtain the following corollary.

\begin{cor}\label{cor2}
    There exists $\epsilon_0>0$ such that for every $\epsilon \in (0,\epsilon_0)$, there exist $R_\epsilon \geq 1$ and $M_\epsilon \geq 1$ such that  $$ \sup_{q \in V_{R_\epsilon,\epsilon}^+}\left\vert \frac{1}{\deg(h)} \log^+\Vert h(q) \Vert - \log^+\Vert q\Vert \right\vert \leq M_\epsilon $$
    for every $h\in \mathcal{H}$.
\end{cor}

\begin{proof}
    The lower bound 
    $$\frac{1}{\deg(h)} \log^+\Vert h(q) \Vert - \log^+\Vert q\Vert \geq \frac{\log C_\epsilon}{\deg(h)}$$
    is given by Proposition \ref{prop6}.
    Moreover, there exists $C\geq 1$ such that 
    $$\Vert h(q )\Vert \leq C \Vert q \Vert ^{\deg(h)}$$
    when $\Vert q \Vert \geq 1$. This gives the desired upper bound.
\end{proof}

\subsection{Green functions}\label{section3.2}
We give the proof of existence of the Green function $G_\omega$ for $\omega=(f_n)_{n\geq 0} $ in a full measure subset of $\Omega$. Let $a_\omega \in \mathcal{A}^\mathbb{N}$ and $e_\omega \in \mathcal{E}^\mathbb{N}$ be the sequences given by Proposition \ref{prop17}. We assume first that the normal form of $f^n_\omega$ starts with an elementary map. We define $h_k:= a_ke_k$ for any $k \geq 1$. Similarly, we define $h_{j,n}:=a_{j,n}e_{j,n}$ for any $j \leq l(n)$. Thus, we can write $f^n_\omega$ as a product of elements of $\mathcal{H}$, that is, $$f^n_\omega=h_{l(n),n} \cdots h_{1,n} h_k \cdots h_1$$
when $n \geq t_k$, up to multiplication on the left by an element of $\mathcal{A}$.

We define also $h^k_\omega:=h_k \cdots h_1 $ and $h^{j,n}_\omega:= h_{j,n}\cdots h_{1,n}$ so that we have the equality $f^n_\omega=h^{l(n),n}_\omega h^k_\omega$ when $n\geq t_k$.
We introduce the following notation for the degrees: we denote by $d_k:=\deg(h_k)$ the degree of $h_k$ and by $d_{j,n}:=\deg(h_{j,n})$ the degree of $h_{j,n}$. We define also $d^k_\omega:=d_k \cdots d_1$ and $d^{j,n}_\omega:=d_{j,n} \cdots d_{1,n}$.
By Lemma \ref{lemme1}, we have
$$\deg(f^n_\omega)=d^k_\omega d^{l(n),n}_\omega.$$

The proof that the sequence $$G_{n,\omega}:= \frac{1}{\deg(f^n_\omega)}\log^+ \Vert f^n_\omega \Vert   $$
converges to a continuous plurisubharmonic function on $\mathbb{C}^2$ goes as follows: we first prove that the auxiliary sequence $$u_k:=\frac{1}{d^k_\omega} \log^+\Vert h^k_\omega \Vert$$
converges to a continuous plurisubharmonic function, and then, we show that the limit of $G_{n,\omega}$ is the same as the limit of $u_k$.

\begin{prop}\label{prop18}
    The sequence $(u_k)_{k\geq 1} $ converges to a continuous plurisubharmonic function $u$ on $\mathbb{C}^2$. Moreover, there exists $\epsilon_0>0$, such that for every $\epsilon \leq \epsilon_0$, there exists $R_\epsilon \geq 1$, satisfying $$V_{R_\epsilon,\epsilon}^+ \subset \{ u >0\}.$$ 
    These constants do not depend on $\omega$.
\end{prop}

\begin{proof}
    Fix $\epsilon \in (0,\epsilon_0)$ where $\epsilon_0$ is given by Proposition \ref{prop6}. Then, for $R$ large enough, we have
    $$h^i_\omega (V_{R,\epsilon}^+) \subset V_{R,\epsilon}^+$$
    for every $i \geq 1$.
    Using Corollary \ref{cor2}, we obtain
    $$\sup_{V_{R,\epsilon}^+}|u_{i+1}-u_i | \leq \frac{M_\epsilon}{d^i_\omega}.$$
    We write $$u_k=u_1+\sum_{i=1}^{k-1} (u_{i+1} - u_i),$$
    and we conclude that $u_k$ converges uniformly to a continuous function $u$ on $V_{R,\epsilon}^+$. Moreover, by applying inductively the inequality of Proposition \ref{prop6}, we obtain that $u$ is positive on $V_{R,\epsilon}^+$ whenever $R$ is large enough. Notice that $R$ and $\epsilon$ do not depend on $\omega$, as the proof works for any sequence $(h_k) \in \mathcal{H}^\mathbb{N}$.
    
    We consider now a point $q \in \mathbb{C}^2.$ Choosing $R$ large enough, we may assume that $\Vert q \Vert \leq R$. To define the limit $u(q)$, we  distinguish two cases. If the orbit $(h_\omega^k(q)) $ remains in $B(0,R)$, we have convergence to $u(q)=0$. Otherwise, there exists a hitting time $l\geq 1$ of the complement of the bidisk $B(0,R)^c$. By Proposition \ref{prop6}.(3), we have $$h^l_\omega(q) \in V_{R,\epsilon}^+.$$
    Then, for $k \geq l$, we write
    $$u_k(q) = \frac{1}{d^l_\omega } \frac{1}{d_k \cdots d_{l+1}} \log^+\Vert h_k \cdots h_{l+1}(h^l_\omega(q))\Vert.$$
    The limit $$v_l:=\lim_{k \rightarrow +\infty} \frac{1}{d_k \cdots d_{l+1} }\log^+\Vert h_k \cdots h_{l+1}\Vert $$
    exists on $V_{R,\epsilon}^+$ by the same proof as above. 
    Thus, the limit of $(u_k(q))$ is well defined and we have $$u(q)=\frac{1}{d^l_\omega}v_l(h^l_\omega(q)).$$
    
    We deduce that $u$ is continuous on the set of points on which $(h^k_\omega(q))$ is unbounded. Indeed, if $q$ is such a point, we can find $l\geq 1$, and a neighborhood of $q$ which is sent by $h^l_\omega$ to $V_{R,\epsilon}^+$. Moreover, there exists $C>0$, such that we have $$u_{k+1} \leq u_k +\frac{C}{d^k_\omega}$$
    on $\mathbb{C}^2$, and this shows that $u$ is a plurisubharmonic function as there exists $(c_k) $ such that the sequence of plurisubharmonic functions $(u_k+c_k)$ is decreasing to $u$. 
    
    It remains to show that $u$ is continuous on the boundary of $\{u=0\}$. Consider a sequence $(q_n )$ that converges to a point $q$ satisfying $u(q)=0$. We may assume $u(q_n) >0$ for every $n$. Denote by $t_n$ the hitting time of $V_{R,\epsilon}^+$ for the orbit $(h^k_\omega(q_n))_{k \geq 1}$.
    Remark that we necessarily have $t_n \rightarrow +\infty.$ We can write
    $$u(q_n)=\frac{1}{d^{t_n}_\omega} v_{t_n}(h^{t_n}_\omega(q_n)).$$
    Moreover, we have $$\sup_{l \geq 1} \sup_{h_l(B(0,R))} v_l < +\infty.$$
    We conclude that $u(q_n) \rightarrow 0$.
\end{proof}

\begin{proof}[Proof of Theorem \ref{thm1}]
We conclude the proof of the existence of $G_\omega$ by showing that the sequence $(G_{n,\omega})$ converges to $u$. We assume first that the normal form of $f^n_\omega$ ends by an affine map. Consider $q \in \mathbb{C}^2$ and choose $R$ large enough so that $\Vert q \Vert \leq R$. As in the proof of Proposition \ref{prop18}, we distinguish two cases. If the orbit $(h^k_\omega (q))$ remains in the bidisk $B(0,R)$, we bound 
   $$|G_{n,\omega}(p) -u_k(q)|=\frac{1}{d^k_\omega}\left|\frac{1}{d^{l(n),n}_\omega} \log^+\Vert h_\omega^{l(n),n}(h^k_\omega(q)) \Vert - \log^+\Vert h^k_\omega(q)\Vert \right|$$
for $n \geq t_k$, using the following:
    $$\sup_{r \geq 1, h_1, \ldots, h_r \in \mathcal{H}} \;\sup_{B(0,R)} \frac{1}{d_r \cdots d_1}\log^+\Vert h_r \cdots h_1 \Vert < +\infty.$$
    Thus, there exists $C>0$ such that 
    $$|G_{n,\omega}(q) -u_k(q)| \leq \frac{C}{d^k_\omega},$$
    and the sequence $(G_{n,\omega}(q))$ is a Cauchy sequence that converges to $u(q)$.
Otherwise, by Proposition \ref{prop6}.(3), there exists a time $l \geq 1$ such that $ h^k_\omega (q)$ is in $ V_{R,\epsilon}^+$ for every $k \geq l$. If $n \geq t_k,$ we have
\begin{align*}
    |G_{n,\omega}(q) -u_k(q)|&=\frac{1}{d^k_\omega}\left|\frac{1}{d^{l(n),n}_\omega} \log^+\Vert h_\omega^{l(n),n}(h^k_\omega(q)) \Vert - \log^+\Vert h^k_\omega(q)\Vert \right|\\
    &\leq \frac{1}{d^k_\omega} \sum_{j=0}^{l(n)-1} \frac{1}{d^{j,n}_\omega} \left|\frac{1}{d_{j+1,n}}\log^+\Vert h_{j+1,n} (q_{j,n})\Vert - \log^+\Vert q_{j,n} \Vert\right|
\end{align*}
where $q_{j,n}:=h_{j,n}\cdots h_{1,n}(h^k_\omega(q))$ is in $ V_{R,\epsilon}^+$.
Using Corollary \ref{cor2}, we obtain
   \begin{align*}
      |G_{n,\omega}(q) -u_k(q)| &\leq \frac{M_\epsilon}{d^k_\omega} \sum_{j = 0}^{l(n)-1} \frac{1}{d^{j,n}_\omega}\\
       &\leq \frac{2M_\epsilon}{d^k_\omega}.
   \end{align*}
This shows that $(G_{n,\omega}(q))$ converges to $u(q)$ and this concludes the proof of existence of $G_\omega$. Notice that the proof works equally well if the normal form of $f^n_\omega$ ends by an elementary map.

If the normal form of $f^n_\omega$ starts by $a \in \mathcal{A}\cup \{1\}$, we thus have shown that $G_\omega$ is positive on the set $$U(a^{-1}(I)):=\bigcup_{0<\epsilon<\epsilon_0} a^{-1}(V_{R_\epsilon,\epsilon}^+).$$ Its closure in $\mathbb{P}^2$ is a neighborhood of $L_\infty \setminus a^{-1}(I)$. This proves the Theorem \ref{thm1}.(2) in which  $F:=\mathcal{A}(I) \cup I$ is the orbit of $I$ under the action of $\mathcal{A} \cup \{1\}$. 
\end{proof}

We conclude by one remark: in the classical setting, the zero locus of $G_f$ coincides with the points that have bounded forward orbit. In the setting of random dynamics, we only have that the zero locus of $G_\omega$ coincides with the points $p \in \mathbb{C}^2$ that have bounded orbit along the sequence $(h^k_\omega)$ which are special points on the geodesic ray $\xi(\omega)$.

\section{Green functions on good models}\label{section4}

In this section, we prove Theorem \ref{thm7} using the following idea. By Theorem \ref{thm1}.(2), the intersection of the closure in $\mathbb{P}^2$ of the zero locus of $G_\omega$ and the line at infinity consists of one point $p(\omega)\in F$. If the points $p(\omega)$ and $p(\omega')$ are different, then we can distinguish the Green functions $G_\omega$ and $G_{\omega'}$. If $p(\omega)=p(\omega')$, we consider the blow-up of this point, and we show that after a finite number of blow-ups, we can find a good model $\pi\colon X \rightarrow \mathbb{P}^2$ on which the closures of the zero loci of $G_\omega$ and $G_{\omega'}$ intersect the divisor at infinity $ X_\infty:=X \setminus \pi^{-1}(\mathbb{C}^2)$ in distinct points. This rely on the study of the link between the properties of the sequence of base points of $f^n_\omega$ and the behaviour of the random walk on the Bass-Serre tree. We will consider elements of $G$ acting as birational maps on models of the projective plane. We conclude this section by the proof of Theorem \ref{thm2}. We refer to \cite{La2} for the basics on birational maps of $\mathbb{P}^2$.

\subsection{Models and resolutions} 
A \textbf{model} of $\mathbb{P}^2$ is a birational morphism $\pi\colon X \rightarrow \mathbb{P}^2$ where $X$ is a smooth projective surface. Any model is obtained by a finite number of blow-ups.  
Recall that a birational map $f$ of $\mathbb{P}^2$ admits a resolution: there exists a model $\pi\colon X \rightarrow \mathbb{P}^2$ such that $\sigma:=f \circ \pi$ is a morphism.
\[
\begin{tikzcd}
    & X \arrow[rd, "\sigma" ] \arrow[ld, "\pi", swap] & \\
    \mathbb{P}^2 \arrow[rr, dashed, "f"] &  &\mathbb{P}^2
\end{tikzcd}
\]
Thus, any birational map can be written as a composition of blow-ups and blow-downs. The \textbf{base points} of $f$ are the points blown-up in a minimal resolution of $f$. Note that $$|\text{base}(f)|=|\text{base}(f^{-1})|$$ for any birational map $f$ of $\mathbb{P}^2$. The number of base points is also subadditive:
$$|\text{base}(g \circ f) | \leq |\text{base}(f)| +|\text{base}(g)|$$
whenever $f, g \in \text{Bir}(\mathbb{P}^2)$. 

\subsection{Birational dynamics of polynomial automorphisms} 
When $f$ is an automorphism of $\mathbb{A}^2_\mathbb{C}$, its minimal resolution has the following structure.

\begin{prop}\label{prop8}
    Let $f \in G$. If $f$ is not affine, then a minimal resolution of $f $ can be decomposed as $$\pi : X=X_r \overset{\pi_r}{\longrightarrow} X_{r-1} \longrightarrow \dots \rightarrow X_1 \overset{\pi_1}{\rightarrow} X_0:=\mathbb{P}^2$$
    where $\pi_1$ is the blow-up of a point $p_1 \in L_\infty$ and, for $i\geq 2$, $\pi_i$ is the blow-up of a point $p_i$ lying on the exceptional divisor $E_{p_{i-1}} \subset X_{i-1}$ of $\pi_{i-1}$. The point $p_1$ is equal to the indeterminacy point $I_f$ of $f$. The strict transform $\pi^{-1}(L_\infty) $ is the first curve to be contracted by $\sigma:=f\circ \pi$, and then, $\sigma$ contracts every exceptional divisor of $\pi$, except the last one $E_{p_r}$ which is sent to $L_\infty$.
\end{prop}

\begin{proof}
    See \cite[Lemma 7.7]{La2}.
\end{proof}

The next proposition shows that $\pi$ determines $\sigma$ up to an affine automorphism and will be a key ingeredient for the proof of Theorem \ref{thm7}.

\begin{prop}\label{prop9}
        Let $f,g \in G$ be elements with the same base points. Then, there exists an affine automorphism $a \in A$ such that $f=ag$. 
\end{prop}

    \begin{proof}
    Let $\pi\colon X \rightarrow \mathbb{P}^2$ be a minimal resolution of $f$. 
    By Proposition \ref{prop8}, the first $(-1)$-curve contracted by $\sigma=f \circ \pi$ is the strict transform $\pi^{-1}(L_\infty)$. Moreover, if we consider a decomposition of $\sigma$ in a product of blow-downs, we see that at each step there are exactly two $(-1)$-curves, except for the last step where there is only one. As the exceptional divisor $E_{p_r}$ is never contracted, the only choice for $\sigma$ is the image in $\mathbb{P}^2$ of the last $(-1)$-curve contracted. Thus, $\pi$ determines $\sigma$ up to an automorphism of $\mathbb{P}^2$. In other words, if $\sigma'=g \circ \pi$, we have that $\sigma' \circ \sigma^{-1}$ is an automorphism of $\mathbb{P}^2$ that fixes $L_\infty$, and we obtain $\sigma' \circ \sigma^{-1} \in A$.
    \end{proof}

In presence of algebraic stability, the sequence of base points behaves well, as shows the following proposition.

\begin{prop}\label{prop7}
        Let $f,g\in G$ be such that $I_g \neq I_{f^{-1}}$. We have $$|\emph{base}(g\circ f)|=|\emph{base}(f)|+|\emph{base}(g)|$$
        and the sequence of base points of $g\circ f$ begins by the base points of $f$.
\end{prop}

\begin{proof}
    See \cite[Proposition 4.19]{La2}.
\end{proof}
    
\subsection{Base points of random products}
Let $\mu $ be a finitely supported, non-elementary and purely loxodromic measure on $G$. We show that, for any $l\geq 1$, the  first $l$ points in the sequence of base points of $f^n_\omega$ do not depend on $n$ when $n$ is large enough.

\begin{prop}\label{prop10}
    For $\mu^\mathbb{N}$-almost every $\omega \in \Omega$, there exists a sequence $(p_l(\omega))_{l \geq 1}$ with $p_1(\omega) \in X_0:=\mathbb{P}^2$, and $p_l(\omega) $ is lying on the exceptional divisor $E_{l-1}$ of the blow-up of $p_{l-1}(\omega)$, satisfying the following: for every $l\geq 1$, there exists $\tau_l(\omega) \geq 1$, such that the first base points of $f^n_\omega$ are $p_1(\omega), \dots, p_l(\omega)$ when $n\geq \tau_l(\omega)$.  
\end{prop}

\begin{proof}
    We can assume that the normal form of $f^n_\omega$ starts by an elementary map and we write $f^n_\omega$ as a product of elements of $\mathcal{H}$ as in Section \ref{section3.2}, $$f^n_\omega=h^{l(n),n}_\omega h_\omega^k$$
    for $n\geq t_k(\omega)$, up to multiplication on the left by an affine map. 
    By Lemma \ref{lemme1}, for any $h_1,h_2 \in \mathcal{H}$, we have $I_{h_1}
     \neq I_{h_2^{-1}}$. Thus, we can apply Proposition \ref{prop7} and we obtain that the sequence of base points of $f^n_\omega$ begins with the sequence of base points of $h_\omega^k$ which does not depend on $n \geq t_k$.
\end{proof}

In particular, the indeterminacy point of $f^n_\omega$ on $\mathbb{P}^2$ is constant for $n \geq \tau_1(\omega)$ and is equal to the point $p_1(\omega)$ given by Theorem \ref{thm1}.(2).

\subsection{Dynamics on good models}

We now study the dynamics of $f^n_\omega$ in a model obtained by blowing-up the points $p_1(\omega),\dots, p_{l}(\omega)$. If $\pi\colon X \rightarrow \mathbb{P}^2$ is a model, we denote by $X_\infty:=X \setminus \pi^{-1}(\mathbb{C}^2)$ the union of the exceptional divisors of $\pi$ and the strict transform $\pi^{-1}(L_\infty)$.

\begin{prop}\label{prop11}
    Let $l\geq 1$. Consider the model $\pi_{l}:X_{l} \rightarrow \mathbb{P}^2$ obtained by the blow-up of the points $p_1(\omega), \dots, p_{l}(\omega)$. Then, the closure of the zero locus of $G_\omega$ in $X_l$ intersects the divisor at infinity in the point $p_{l+1}(\omega)$, that is,  $$\overline{\{G_\omega=0\}} \cap  X_{l, \infty}=p_{l+1}(\omega).$$
\end{prop}

\begin{proof}
    We can assume $p_1(\omega)=I$, that is, the normal form of $f^n_\omega $ starts with an elementary map, and we write $$f^n_\omega= h^{l(n),n}_\omega h_\omega^k$$
    for $n\geq t_k$, up to multiplication by an affine map on the left. We choose  $k$ large enough, so that $p_{l+1}(\omega)$ is a base point of $h^k_\omega$.
    Denote by $\pi_k : X_k \rightarrow \mathbb{P}^2$ a minimal resolution of $h^k_\omega$. We can write $\pi_k= \pi_l \circ \mu$ where $\mu : X_k \rightarrow X_l$ is a birational morphism that contracts at least one curve.
    By Proposition \ref{prop8}, the birational map $\sigma_k:=h^k_\omega \circ \pi_k$ contracts every exceptional divisor of $\pi_k$ except the last one to the point $h^k_\omega(I)$ which is distinct from $I$. Thus, any point $q \in X_{l,\infty} \setminus p_{l+1}(\omega)$ has a neighborhood $V_{l}(q)$ which is sent by $h^k_\omega \circ \pi_{l}$  into $V_{R,\epsilon}^+$. We obtain that for all $x \in V_l(q)$,  $$G_\omega(\pi_{l}(x))>0$$
    and this concludes the proof.
\end{proof}

The next proposition, combined with Proposition \ref{prop11} concludes the proof of Theorem \ref{thm7}.

\begin{prop}\label{prop12}
    The sequences of base points $(p_l(\omega))_l $ and $(p_l(\omega'))_l$ are different for $\mu^\mathbb{N}\otimes \mu^\mathbb{N}$-almost every $(\omega,\omega')$.
\end{prop}

\begin{proof}
    Consider a sequence $(p_l)_{l \geq 1}$, and suppose that the set $$R:= \{\omega \in \Omega \, | \;\forall l \geq 1, \, p_l(\omega)=p_l  \} $$
    has positive measure $\mu^ \mathbb{N} (R) >0$.
    We can assume $p_1=I$. The idea is to use Proposition \ref{prop9} to show that the walk at time $n$ is contained in a set of small cardinality, which contradicts the non-triviality of the Poisson boundary. Consider the right walk of reflected law $\check \mu$. Applying normal forms for right products, we can write $$r^n_\omega =h_{\omega}^{k(n)} h^{l(n),n}_\omega, $$
    up to multiplication by an affine map on the right. Here, $k(n)$ is the largest integer $k$ such that $n \geq  t_k$. We can read $l(n)$ as the distance from the walk to its geodesic ray, that is, $$l(n) =2d(r^n_\omega, \xi(\omega)).$$
    By \cite[Theorem 1.3]{MT1}, the random walk on the Bass-Serre tree satisfy logarithmic geodesic tracking, that is, up to consider a subset of positive measure of $R$, there exists $C\geq 1$ such that $$l(n) \leq C \log(n)$$
    for every $\omega \in R$.

    By subadditivity of the number of base points, there exists $B\geq 1$, such that $$|\text{base}(h^{k(n)}_\omega)| \leq Bn $$
    for every $\omega \in R$. If $\omega,\omega'\in R $ satisfy $|\text{base}(h^{k(n)}_{\omega})|=|\text{base}(h^{k(n)}_{\omega'})|$, then their inverses $(h^{k(n)}_{\omega})^{-1}$ and $(h^{k(n)}_{\omega'})^{-1}$ have exactly the same base points. Thus, by Proposition \ref{prop9}, there exists $a \in \mathcal{A}$, such that $$h^{k(n)}_{\omega'}=h^{k(n)}_\omega a.$$
    In this case, $r^n_\omega$ and $r^n_{\omega'}$ differ by a product of at most $l(n)$ elements of $\mathcal{H}$.  
    
    We obtain that the walk at time $n$ is located in a subset $A_n\subset G$ satisfying  $$|A_n| \leq Bn|\mathcal{H}|^{C\log(n)}.$$
    
    This contradicts Corollary \ref{cor3}.
\end{proof}

\subsection{Stationary measures are compactly supported}
We now prove Theorem \ref{thm2}. By Proposition \ref{prop12}, we can consider $l\geq 1$ such that $p_{l}( \omega)$ is not constant. We assume that $l$ is minimal for this property, that is, $p_1, \dots, p_{l-1}$ are constant almost surely. Consider $\pi\colon X \rightarrow \mathbb{P}^2$ the blow-up of the points  $p_1, \dots, p_{l-1}$. Let $p_{A_1}$ and $p_{A_2}$ be distinct points on $X_\infty$ such that the subsets $A_i=\{\omega \in \Omega \mid \, p_l(\omega) = p_{A_i}\}$ have positive measure for $i\in \{1,2\}$.

\begin{prop}\label{prop13}
    There exists a neighborhood $V_\infty $ of $X_\infty$, and a neighborhood $V_i\subset V_\infty$  of $p_{A_i}$ for $i\in \{1,2\}$, such that
    \begin{enumerate}
        \item $V_1 \cap V_2= \emptyset,$
        \item for every $q \in \pi^{-1}(\mathbb{C}^2) \cap (V_\infty \setminus V_i)$  and $\omega \in A_i$,
        $$G_\omega(\pi(q)) >0.$$
    \end{enumerate}
\end{prop}

\begin{proof}
    We can assume $l>1$ and $p_1=I$.
    Consider $k \geq 1$ satisfying $$\min_{h_1, \dots, h_k \in \mathcal{H}^k} \sum_{i=1}^k |\text{base}(h_i)| > l,$$
    so that the sequence of base points of any product of $k $ elements of $\mathcal{H} $ contains at least $l$ base points.
    Let $\omega \in A_i$. By the proof of Proposition \ref{prop11}, there exist a neighborhood $V_\infty$ of $X_\infty$ and a neighborhood $V_{i}\subset V_\infty$ of $p_{A_i}$ such that $G_\omega$ is positive on $\pi (V_\infty \setminus V_{i})$. There are only finitely many products $h_k \dots h_1$, so we can choose $V_\infty $ and $V_{i}$ that works for every $\omega \in A_i$. Moreover, we can choose $V_{1}$ and  $V_{2}$ to be disjoint.
\end{proof}

\begin{proof}[Proof of Theorem 1.1]
Let $\nu$ be a $\mu$-stationary ergodic probability measure on $\mathbb{C}^2$. We consider the dynamics on $X\setminus X_\infty$ via the isomorphism with $\mathbb{C}^2$ induced by $\pi$. If $\nu$ is not compactly supported, there exists  $U \subset V_\infty$ of positive measure. We can assume that $U $ is bounded and $\nu(U \setminus V_{1}) >0$. If $q$ is a point in $U\setminus V_{1}$ and $\omega \in A_1$, then by Proposition \ref{prop13}, the orbit $f^n_\omega(q)$ leaves every compact subset of $X$ and thus visits $U\setminus V_{1}$ only finitely many times. This contradicts Birkhoff ergodic theorem.
Thus, every ergodic $\mu$-stationary measure is supported on $\mathbb{C}^2\setminus \pi(V_\infty)$. We obtain the result for any stationary measure by ergodic decomposition.
\end{proof}

\section{Stiffness}

We now prove Theorems \ref{thm3} and \ref{thm4} by combining Roda's results \cite{Roda} and Theorem \ref{thm7}. We conclude this section by the proof of Theorem \ref{thm5}. We also state further results concerning periodic orbits and the non purely loxodromic case. We refer to \cite{LQ} for the basics on the theory of smooth random dynamical systems.

\subsection{Examples} We can construct an example of a non-elementary and purely loxodromic $\mu$ with no $\mu$-stationary measures as follows. If $f$ is a loxodromic automorphism, we denote by $K_f\subset \mathbb{C}^2$ the compact set of points with bounded backward and forward orbit. Recall that if $\Sigma_\mu$ is symmetric, then the support of any $\mu$-stationary measure is a $\Gamma_\mu$-invariant compact set. Thus, if $f$ and $g$ are two loxodromic automorphisms generating a non-elementary, purely loxodromic subgroup, such that $K_f$ and  $K_g$ do not intersect, then for any measure $\mu$ supported on $f, g $ and their inverses, there exists no $ \mu$-stationary measure. If $f$ and $g$ are loxodromic automorphisms generating a non-elementary subgroup, it suffices to conjugate $g$ by a large translation $g':=\tau g \tau^{-1}$ to obtain that $K_f $ and $K_{g'}$ are disjoint.

It is also not difficult to construct examples of non-elementary, purely loxodromic measures having a finite orbit. Indeed, it suffices to take a ping-pong group as in Section \ref{section2.2}, with $f$ a Hénon map having a fixed point at the origin. 

To our knowledge, there is no example of a non-elementary, purely loxodromic and non-dissipative measure $\mu$ that admits a non-atomic $\mu$-stationary measure.  

\subsection{Lyapunov exponents}
Recall that for $f \in G$, we denote by $\mathrm{Jac}(f)\in \mathbb{C}$ the complex Jacobian of $f$, defined by $f^*dx\wedge dy=\mathrm{Jac}(f) dx\wedge dy$. We say that a finitely supported measure $\mu$ on $G$ is \textbf{non-dissipative} if $$\sum_{f \in \Sigma_\mu} \mu(f) \log |\mathrm{Jac}(f) | \geq 0.$$
As explained in the introduction, we cannot expect stiffness to always hold when $\mu$ is dissipative. 

The support of $\mu$ is finite, and as the support of $\nu$ is compact by Theorem \ref{thm2}, the Lyapunov exponents of $\nu$,
$$\lambda^+:=\lim \frac{1}{n}\log \Vert d_qf^n_\omega \Vert, \quad \lambda^-:=\lim \frac{1}{n}\log \Vert (d_qf^n_\omega)^{-1}\Vert^{-1}$$
are well-defined for $\mu^\mathbb{N}\otimes \nu$-almost every $(\omega,q) \in \Omega \times \mathbb{C}^2 $, and we have $$\lambda^+ + \lambda^- = \sum_{f \in \Sigma_\mu} \mu(f) \log |\mathrm{Jac}(f)|.$$
Thus  $\lambda^+ + \lambda^- \geq 0$ in the non-dissipative case. If the two exponents $\lambda^+$ and $\lambda^- $ are non-negative, the invariance principle of Crauel \cite[Theorem 5.1]{Cr} and Avila-Viana \cite[Theorem B]{Av} allows us to conclude that $\nu$ is $\mu$-invariant. Otherwise, the stationary measure $\nu$ is \textbf{hyperbolic}: $\lambda^+>0>\lambda^-.$
\subsection{Stable manifolds}
Consider a hyperbolic ergodic $\mu$-stationary measure $\nu$. The stable manifolds 
$$W^s(\omega,q) := \{q'\in \mathbb{C}^2\, | \limsup \frac{1}{n} \log d(f^n_\omega(q),f^n_\omega(q')) <0\} $$
are complex immersed submanifolds for $\mu^\mathbb{N}\otimes \nu$-almost every $(\omega,q)$. They are biholomorphic to $\mathbb{C}$ by \cite[Proposition 7.8]{CD1}. The next proposition shows that stable manifolds are not non-random and is a key ingredient for applying Roda's theorem.

\begin{prop}\label{prop15}
    For $\mu^\mathbb{N}\otimes \nu$-almost every $(\omega,q)$ and $\mu^\mathbb{N}$-almost every $\omega'$, we have $W^s(\omega, q) \neq W^s(\omega',q).$
\end{prop}

\begin{proof}
 If a point $q$ is in the support of $\nu$, then by compactness of the support of $\nu$, we have  $$W^s(\omega,q)\subset \{G_\omega=0\}.$$ 
 Moreover, as $W^s(\omega,q)$ is biholomorphic to $\mathbb{C}$, its closure in any model $X_l$ intersects the divisor at infinity $X_{l,\infty}$. Thus, by Proposition \ref{prop11} we have $W^s(\omega,q) \cap  X_{l,\infty}= p_{l+1}(\omega)$.
 We conclude by Proposition \ref{prop12}.
\end{proof} 

\begin{proof}[Proof of Theorems \ref{thm3} and \ref{thm4}]
We now apply Roda's theorem \cite{Roda}. Note that \cite[Theorem 1.1.1]{Roda} is given in the setting of random dynamical systems on compact complex surfaces, but the result can be extended to compactly supported stationary measures on $\mathbb{C}^2$. We only need to replace \cite[Corollary 3.3.2]{Roda} by Proposition \ref{prop15}.    
\end{proof}

\subsection{Classification of invariant measures under positivity of the entropy}

The proof of Theorem \ref{thm5} is a combination of Theorem \ref{thm4} with results due to Cantat--Dujardin \cite[Section 11]{CD1} and Brown--Rodriguez Hertz \cite[Section 13.2.4]{BRH}. The main part of the proof is contained in the article of Cantat-Dujardin, so we only sketch the main ideas. Up to replace $f $ by a power $f^k$, we can consider a measure 
$$\mu_{a,b,c,d}:= a\delta_f + b\delta_{f^{-1}}+c \delta_g + d \delta_{g^{-1}}$$
where $f $ and $g$ generate a non-elementary, purely loxodromic subgroup of $G_\nu$. The measure $\nu$ is $G_\nu$-invariant, and thus it is a $\mu_{a,b,c,d} $-stationary, compactly supported measure on $\mathbb{C}^2$. Using ideas of Brown and Rodriguez Hertz, Cantat--Dujardin prove that there exist $a,b,c,d$ such that $\nu$ is hyperbolic as a $\mu_{a,b,c,d} $-stationary measure. We can also assume that the sum of the Lyapunov exponents of $\nu$ is non-negative. Thus, we can apply Theorem \ref{thm4} to conclude.  

\subsection{Periodic points}
We say that a point $q \in \mathbb{C}^2$ is $\Gamma$-periodic if it has finite orbit: $|\Gamma(q)|<\infty$. The next result shows that there are only finitely many periodic orbits.

\begin{prop}
    If $\Gamma$ is a finitely generated non-elementary subgroup of $G$, then there are only finitely many $\Gamma$-periodic points.
\end{prop}

\begin{proof}
Suppose first that the set of $\Gamma$-periodic points is Zariski-dense. By \cite[Lemma 6.5]{DF}, we can assume that $\Gamma$ contains two regular automorphisms $f$ and $g$ that generate a non-elementary subgroup and that share a Zariski-dense set of common periodic points. Thus, we obtain a contradiction by \cite[Theorem B]{Abboud2}. Remark that the set of $\Gamma$-periodic points is $\Gamma$-invariant. If the Zariski closure of this set is a curve $\mathcal{C}$, then the intersection of $\gamma(\mathcal{C})$ with $ \mathcal{C}$ is infinite for any $\gamma \in \Gamma$. Thus, Bezout's theorem implies that there exist a curve which is preserved by a finite-index subgroup of $\Gamma$. This contradicts the fact that a loxodromic automorphism of $\mathbb{C}^2$ does not preserve any curve \cite{BS}.     
\end{proof}

\subsection{The non-purely loxodromic case}
Beyond the purely loxodromic case, that is, if the group generated by $\Sigma_\mu$ contains an elementary automorphism, we do not know how to prove compactness of the support of stationary measures. Nevertheless, if one is interested in compactly supported stationary measures, then the classification problem is easier in this setting. Indeed, if $\Sigma_\mu$ is symmetric and contains an elementary map $e$ with $K_e$ finite, then any compactly supported stationary measure is finitely supported. See \cite[Section 6]{FM} for a classification of elementary maps.

\end{document}